\documentclass[12pt]{amsart}

\usepackage{amssymb, amsmath, amsthm, amsfonts}
\usepackage{hyperref}
\usepackage{xcolor,graphicx}
\usepackage{changes}
\theoremstyle{plain}
\newtheorem{theorem}{Theorem}[section]

\newtheorem{lemma}[theorem]{Lemma}
\newtheorem{corollary}[theorem]{Corollary}
\newtheorem{proposition}[theorem]{Proposition}

\theoremstyle{definition}

\newtheorem{remark}[theorem]{Remark}
\newtheorem{example}[theorem]{Example}

\def\h{\mathbb H}
\def\r{\mathbb R}
\def\c{\mathbb C}

\begin{document}
\title[Curve shortening flow in the hyperbolic plane]{Geometric aspects of the curve shortening flow in the hyperbolic plane }
\author{Ivan Krznari\'{c}}
\address{Faculty of Economics and Business. University of Zagreb, Trg John F. Kennedy 6, 10 000 Zagreb, Croatia}
\email{ikrznaric@efzg.hr}
\author{Rafael L\'opez}
\address{Departamento de Geometr\'{\i}a y Topolog\'{\i}a Universidad de Granada 18071 Granada, Spain}
\email{rcamino@ugr.es}


\begin{abstract}
We define a new notion of translations in the hyperbolic plane and explicitly solve the equation of the curve shortening flow. Next, we consider the class of ancient convex solutions and solve the equation of the curve shortening flow when the curvature function is given by separation of variables. Lastly, we prove some area estimates for closed ancient solutions of the curve shortening flow. 
\end{abstract}

\maketitle

 \section{Introduction}

 Let $\gamma\colon I\to\h^2$ be a curve parametrized by arc-length in the hyperbolic plane $\h^2$.   A one-parameter family  of parametrized curves $\{\gamma_t\colon I\to\h^2\}_{t\in [0, T\rangle}$ is said to be the \textit{curve shortening flow} of $\gamma$ if 
 \begin{equation}\label{flow}
 \left\{
 \begin{split}
  \left(\partial_t \gamma_t\right)^\perp&= \kappa(t, \cdot) N(t, \cdot) \\ 
 \gamma_0&=\gamma   \end{split}\right.
\end{equation}
 where    $\kappa(t,\cdot)$ and $N(t,\cdot)$ are the geodesic curvature and the unit normal vector of $\gamma_t$ and $ \left(\partial_t \gamma_t\right)^\perp$ is the normal part of velocity $  \gamma_t$. The theory of the CSF in the Euclidean plane $\r^2$ was first developed in the 1980s by Gage and Hamilton (\cite{gh,ga}) revealing that its  study is non-trivial. This is due to    the difficulty of calculating explicit solutions to the flow and also due to the development of singularities during the evolution. Therefore, it is convenient to study the solutions of the CSF within some specific class of curves. Common examples of such classes are solitons, translating solutions and ancient convex solutions. Although the CSF can be studied on arbitrary surfaces, the most interesting situations, besides the Euclidean plane, arise when studying the evolution of curves in the models of the remaining two $2$-dimensional geometries: the  sphere $\mathbb{S}^2$ and the hyperbolic plane $\h^2.$ The curve shortening flow on $\mathbb{S}^2$ has been considered via various approaches (\cite{blz,bl}). 
 
 To the best of our knowledge, the hyperbolic plane $\h^2$ is the ambient in which the curve shortening flow is studied the least. In  \cite{dt,wx},   the authors have classified the solutions which evolve by isometries during the flow and have calculated some explicit examples of such curves using their characterization.   The mean curvature flow in higher dimensions has been   studied in    \cite{an,cm1,cm2}. Following ideas of  \cite{cm1}, in \cite{wy} a modified flow of \eqref{flow} has been studied preserving the length of the curves.

The purpose of this paper is to give some progress of the CSF in $\h^2$ emphasizing some of its geometric aspects. For example, in Euclidean plane $\r^2$, the study of the CSF by geodesic translations reduces to the study of solitons of the CSF because Eq. \eqref{flow} is invariant by translation. However the situation in $\h^2$ is different. The class of solitons appears in a natural way when we study those solutions of the CSF invariant by a one-parameter groups of isometries. In this paper we are also interested in the solutions obtained by geodesic translations. We will see that such solutions are curves in $\h^2$  with constant curvature. Other interest in this paper is the role of the angle parameter.  Unlike in the Euclidean plane, it is not possible to geometrically and globally define the angle formed by tangent vectors with a fixed direction. Grayson is able to define a global angle where a `density' is added at each point. Motivated by the paper \cite{dhs}, we obtain the solutions of the CSF when $\kappa^2$ is given by separation of variables, proving that besides solitons, the solutions are the curves with constant curvature.

This paper is organized as follows. In Section \ref{sec2} we introduce the notation of the one-parameter groups of isometries in $\h^2$ and we give examples of CSF for curves with constant curvature. In Section \ref{sec3} we turn our attention to the study of  solutions of the CSF by geodesic translations, proving that the initial curve is a curve with constant curvature. We explicitly find the CSF. In Section \ref{sec4} we discuss when the angle can be considered as a parameter in the flow. By taking the natural angle $\varphi$ defined by Grayson \cite{gr}, we find the solutions of the CSF when the pressure function $\kappa^2$ is given by separation of variables $\kappa^2(t,\varphi)=a(t)+b(\varphi)$. Besides the trivial solutions, which do not depend on $t$, we prove in Thm. \ref{t42} that the solution of the CF is a family of curves with constant curvature, that is, circle,  horocycles or equidistant lines. We also find the evolution of the Euclidean angle $\theta$ which can be defined in the upper half-plane model of the hyperbolic plane. Finally in Sect. \ref{sec5}, we prove that the limit of the area enclosed by ancient simply closed convex solution to the CSF is infinity (Thm. \ref{t51}). In particular, the time interval of the flow is $(-\infty,0]$. 

\section{Examples of CSF in $\h^2$} \label{sec2}

Consider the upper half-plane model of the hyperbolic plane $\h^2$, that is,   $\r^2_+= \left\{ (x,y) \in \r^2\colon y > 0\right\}$ endowed with the   Riemannian metric $\langle,\rangle = \frac{dx^2 + dy^2}{y^2}$. Let $\r^2_0=\{(x,0)\colon x\in\r\}$ the horizontal line of equation $y=0$. The ideal boundary  $\h^2_\infty$ is the one-point compactification of $\r^2_0$, that is,  $\h^2_\infty=\r^2_0\cup\{\infty\}$. Identifying $\r^2_+$ with the upper complex plane $\c^+=\{z=a+ib\in\c\colon b>0\}$, the group of orientation-preserving isometries of $\h^2$   coincides with the M\"obius transformations $T(z) = \frac{az+b}{cz+d}$ with real coefficients such that $ad-bc=1$.   There are three one-parameter subgroups of isometries in $\h^2$ which are described as follows:
\begin{enumerate}
\item Hyperbolic translations. We have $\mathcal{H}=\{H_t\colon t\in\r\}$, where $H_t(z) = e^{t}z$.
\item Parabolic translations. We have $\mathcal{P}=\{P_t\colon t\in\r\}$, where $P_t(z) = z + t$.
\item  Rotations. We have $\mathcal{R}=\{ R_t\colon t\in\r\}$, where $R_t(z) =    \frac{\cos(t)z - \sin(t)}{\sin(t)z +\cos(t)}$.
\end{enumerate}

These three subgroups allow us to find a basis of Killing vector fields of $\h^2$. The Killing vector fields are obtained by 
$$
X_H(z)=\frac{d}{dt}{\Big|}_{t=0}H_t(z),\quad X_P(z)=\frac{d}{dt}{\Big|}_{t=0}P_t(z),\quad X_R(z)=\frac{d}{dt}{\Big|}_{t=0}R_t(z).
$$
A simple computation gives 
\begin{equation}\label{kv}
\begin{split}
X_H(x,y)&=x\partial_x+y\partial_y,\\
 X_P(x,y)&=\partial_x,\\
 X_R(x,y)&=-(1+x^2-y^2)\partial_x+2xy\partial_y.
 \end{split}
\end{equation}
Notice that the vector field $X_R$ can be written as $X_R=-X_P-(x^2-y^2)\partial_x+2xy\partial_y$.  

Let $\gamma\colon I \rightarrow \h^2, \gamma = \gamma(u),$ be a parametrized curve where $u$ is an arbitrary parameter. If $\gamma(u) = (x(u), y(u))$,  its   curvature  $\kappa$  is given by 
\begin{equation}\label{ke0}
\kappa =  y\frac{x'y'' - x''y'}{(x'^2 + y'^2)^{3 \over 2}}+\frac{x'}{\sqrt{x'^2+y'^2}}.
\end{equation}
There is a relationship between the curvature $\kappa$ of $\gamma$ and the Euclidean curvature $\kappa_e$ of $\gamma$ when $\gamma$ is considered as a curve in $\r^2_{+}$ endowed with the Euclidean metric. This relation is 
\begin{equation}\label{ke}
\kappa=y\kappa_e+\frac{x'}{\sqrt{x'^2+y^2}}.
\end{equation}
This identity is also a consequence of \eqref{ke0}. 

We know how to describe the curves in $\h^2$ with constant curvature. In the upper-half plane model, these curves are   the intersections of circles and straight-lines of $\r^2$ with $\r^2_+$.  The geodesics of $\h^2$ are given by vertical half-lines from $\r^2_0$ and half-circles with center at $\r^2_0$. All other curves in $\h^2$ with constant curvature are the following:
\begin{enumerate}
\item Circles.  A circle   is the locus of all points a fixed distance from a given point. The distance is called the radius of the circle and the given point is its  center.  From the Euclidean viewpoint, they are Euclidean circles contained in $\r^2_+$.  If the radius is $R$, the   curvature is $\kappa =\coth(R)$. Notice that $\kappa>1$.
\item Horocycles. A horocycle is either a circle tangent to the line $\r^2_0$ or a  line parallel to $\r^2_0$.  Its   curvature is $\kappa=1$. 
\item  Equidistant lines. An equidistant line is the set of points in $\h^2$ equidistant from a given geodesic.  Equidistant lines are  (non-geodesics)  arcs of circles and (non-vertical) half-lines from points of $\r^2_0$. If the equidistant line is of equation $y=mx+n$, then its curvature is $\kappa=1/\sqrt{1+m^2}$.  The   curvature $\kappa$ is a number between $0$ and $1$. 
\end{enumerate}

We give examples of CSF for curves with constant curvature. See also \cite[p. 76--77]{gr}.

\begin{example}[Flow of geodesics]
The evolution under CSF of a geodesic $\gamma\colon \r \rightarrow \h^2$ is  given by $\gamma(t,s) = \gamma(s)$, $ t \in \r$. This is because $\kappa=0$.
\end{example}

\begin{example}[Flow of circles]
\label{ex:circle}A circle of  center     $(a,b) \in \r^2_+$ and radius $R$ coincides with the  the Euclidean circle with center at   $(a, \cosh(R))$ and radius   $b\sinh(R)$.  A parametrization $\gamma$ of 
  a circle of   radius $R$ and   center $(0,1)$  is 
$$\gamma(s)=(0,\cosh(R))+\sinh(R)\left(\cos(s),  \sin(s)\right).$$
 Analogously to the Euclidean case, an initial guess   for its flow is the family of all circles centered at $(0,1)$. Therefore, let 
 $$\gamma_t(s) =(0,  \cosh(r(t)) )+\sinh(r(t))(\cos(s), \sin(s)).$$ 
 We compute all of the quantities of \eqref{flow}. We know $\kappa=\coth(r(t))$ and 
 \begin{equation*}
 \begin{split}
 N&=-( \cosh(r(t)) + \sinh(r(t))\sin(s))(\cos(s),\sin(s)),\\
 \partial_t\gamma_t&=r'(\cosh(r(t))\cos(s),\sinh(r(t))+\cosh(r(t))\sin(s)).
 \end{split}
 \end{equation*}
  Then
 $$\langle( \partial_t\gamma_t)^\perp,N\rangle=-r'(t).$$
 Thus \eqref{flow} reduces into 
 \begin{equation*}
 \left\{
 \begin{split}
  r' &= -\coth(r),\\
   r(0) &= R \end{split}
   \right.
   \end{equation*}
    The solution   is the function $r(t) = \cosh^{-1}(\cosh(R)e^{-t})$.
    Since $\cosh^{-1}$ is defined for numbers  bigger than or equal to $1$, we get $t_{\text{max}} = \ln(\cosh(R))$. At the maximal time we have $r = 0$ so the circles contract to a point in finite time under the CSF.
\end{example} 

\begin{example}[Flow of horocycles]
\label{ex:circle2}
Let $\gamma$ be the horocycle 
$$\gamma(s)=(0,R)+R\left( \cos(s), \sin(s)\right).$$
The ideal boundary  of $\gamma$ is $(0,0)$. Since we guess that the CSF is formed by all horocycles with the same ideal boundary point $(0,0)$, let  
 $$\gamma_t(s) = (0,r(t))+r(t)(\cos(s),  \sin(s)).$$ 
 Now we have $\kappa=1$ and  
 \begin{equation*}
 \begin{split}
 N&=-( r(t)(1+\sin(s))(\cos(s),\sin(s)),\\
 \partial_t\gamma_t&=r'(\cos(s),1+\sin(s)).
 \end{split}
 \end{equation*}
  Therefore
 $$\langle( \partial_t\gamma_t)^\perp,N\rangle=-\frac{r'(t)}{r(t)}.$$
 Then \eqref{flow} now is 
 \begin{equation*}
 \left\{
 \begin{split}
r'&= -r,\\
   r(0) &= R \end{split}
   \right.
   \end{equation*}
    The solution   is   $r(t) = Re^{-t}$.
  In this case, we have  $t_{\text{max}} =\infty$ and $\lim_{t\to\infty}r(t)=0$. Thus the horocycles  contract to   the ideal boundary point $(0,0)$ under the CSF.
\end{example}

\begin{example}[Flow of equidistant lines]
\label{ex:circle3}
Let $0<c<R$. Let $\gamma$ be the equidistant line 
$$\gamma(s)=(0,c)+R\left( \cos(s), \sin(s)\right).$$
From the Euclidean viewpoint, $\gamma$ parametrizes the Euclidean circle centered at $(0,c)$ with radius $\sqrt{R^2-c^2}$. The   curvature is $\kappa=c/R$. Again, we guess that the CSF of such  an equidistant line is formed out of all the equidistant lines in $\h^2$ with the same ideal boundary as the initial curve. Thus, let  $0<c(t)<r(t)$ with $r(t)^2-c(t)^2=R^2-c^2:=k^2,$$k>0$. Define
 $$\gamma_t(s) = (0,\sqrt{r(t)^2-k^2})+r(t)(\cos(s),  \sin(s)).$$ 
 Now we have $\kappa=\frac{\sqrt{r^2(t)-k^2}}{r(t)}$ and  
 \begin{equation*}
 \begin{split}
 N&=-(\sqrt{r(t)^2-k^2}+r(t)\sin(s))(\cos(s),\sin(s)),\\
 \partial_t\gamma_t&=r'(\cos(s),\frac{r}{\sqrt{r^2-k^2}}+\sin(s)).
 \end{split}
 \end{equation*}
  Then
 $$\langle( \partial_t\gamma_t)^\perp,N\rangle=-\frac{r'(t)}{\sqrt{r(t)^2-k^2}}.$$
 Thus \eqref{flow} is 
 \begin{equation*}
 \left\{
 \begin{split}
r'&= -\frac{r^2-k^2}{r},\\
   r(0) &= R \end{split}
   \right.
   \end{equation*}
    The solution   is the function $r(t) = \sqrt{k^2+(R^2-k^2)e^{-2t}}$.
  In this case, we have  $t_{\text{max}} =\infty$ and $\lim_{t\to\infty}r(t)=k$. Thus the equidistant lines $\gamma_t$  contract to the geodesic $s\mapsto k(\cos(s),\sin(s))$, which coincides with the geodesic to which the initial curve $\gamma$ is equidistant to.  
\end{example} 

As a consequence of the examples above, we conclude the following result. 

\begin{proposition}
The limit of the curve shortening flow   
\begin{enumerate}
\item of a  circle is its center;
\item of a horocycle is its ideal boundary point;
\item of an equidistant line is the geodesic to which the initial curve is equidistant to.
\end{enumerate}
\end{proposition}

We finish this section by finding  all the solutions to the CSF that evolve by the three types of one-parameter group of isometries of $\h^2$, namely, $\mathcal{H}$, $\mathcal{P}$ and $\mathcal{R}$.  These solutions will be called {\it solitons}. Here we will use the notation of the Killing vector fields \eqref{kv}.

\begin{proposition} \label{prs}
Let $\gamma $ be a curve in $\h^2$.  If under CSF $\gamma$ evolves by:
\begin{enumerate}
\item Hyperbolic translations of  $\h^2$, then $\gamma$ satisfies
\begin{equation}\label{sh}
\kappa=\langle N,X_H\rangle,\quad \mbox{(hyperbolic soliton)}.
\end{equation}

\item Parabolic translations of  $\h^2$, then $\gamma$ satisfies  
\begin{equation}\label{sh2}
\kappa=\langle N,X_P\rangle,\quad \mbox{(parabolic soliton)}.
\end{equation}

\item Rotations of  $\h^2$, then $\gamma$ satisfies
\begin{equation}\label{sh3}
\kappa= \langle N,X_R\rangle,\quad \mbox{(rotational soliton)}.
\end{equation}
\end{enumerate}
\end{proposition}

\begin{proof}
\begin{enumerate}
\item The flow by hyperbolic translations is given by $\gamma_t(s)=e^t\gamma(s)$. If  $\gamma(s)=(x(s),y(s))$, then  $N(t,s)=e^t N(s)$ and  $\kappa(t,s)=\kappa(s)$. Notice that $N=\frac{y}{\sqrt{x'^2+y'^2}}(-y',x')$. Since $\partial_t\gamma_t(s)=e^t\gamma$, then \eqref{flow} is simply $\kappa= \langle N,\gamma\rangle$.  
 \item Let $\gamma(s)=(x(s),y(s))$ and consider the flow by parabolic translations  $\gamma_t(s)=\gamma(s)+t(1,0)$. Again $N(t,s)=N(s)$ and $\kappa(t,s)=\kappa(s)$. Since $\partial_t\gamma_t(s)=(1,0)$, then \eqref{flow} is $\kappa=\langle N,(1,0)\rangle$.
 \item 
If $\gamma(s)=(x(s),y(s))$, the flow by   rotations is given in complex notation by 
$$\gamma_t(s)=\frac{\cos(t)\gamma(s)-\sin(t)}{\sin(t)\gamma(s)+\cos(t)}.$$
Since $R_t$ are isometries of $\h^2$, then $\kappa(t,s)=\kappa(s)$.  If the coordinates of $\gamma_t(s)$ area $((\gamma_t(s))_x,(\gamma_t(s))_y)$, the unit normal vector is 
$$N(t,s)=\frac{(\gamma_t(s))_y}{\sqrt{x'^2+y'^2}}\left(\begin{array}{ll}
A& B \\ B&-A\end{array}\right)\left(\begin{array}{l}
-y' \\ x' \end{array}\right)$$
where 
$$A= y^2\sin(t)^2-x(s)( x\sin (t)+\cos (t))^2,\quad B=2\sin(t) y( x\sin (t)+\cos (t)).$$
    Then a computation  yields \eqref{sh3}.  
    \end{enumerate}
    \end{proof}

\begin{figure}[hbtp]
\begin{center}
\includegraphics[width=.3\textwidth]{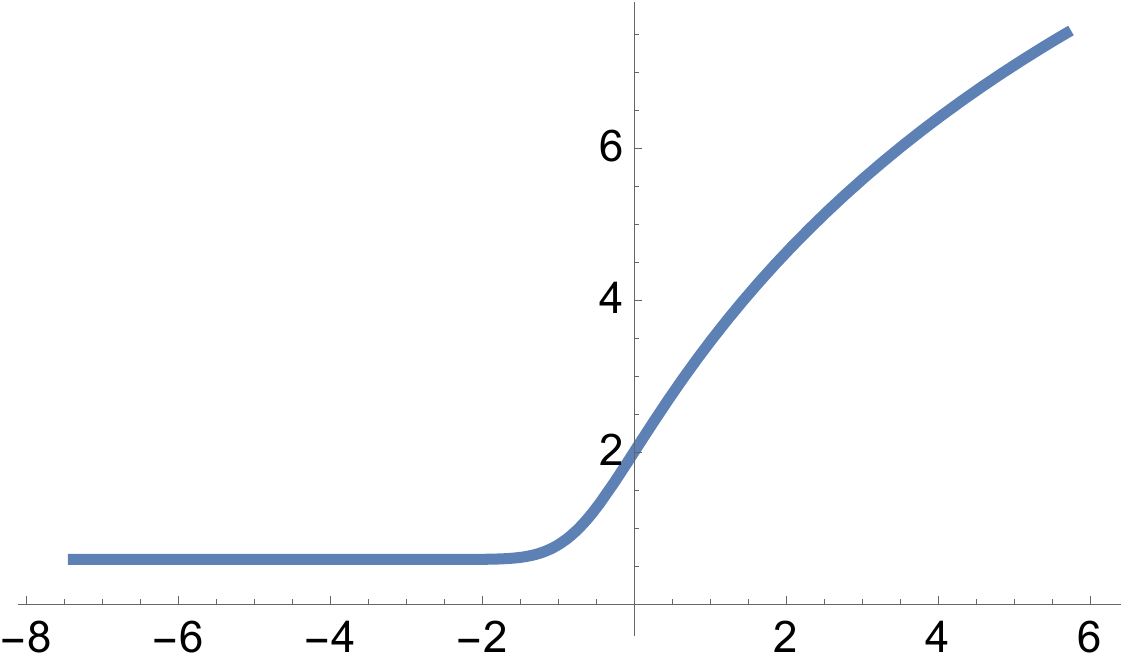},\quad \includegraphics[width=.3\textwidth]{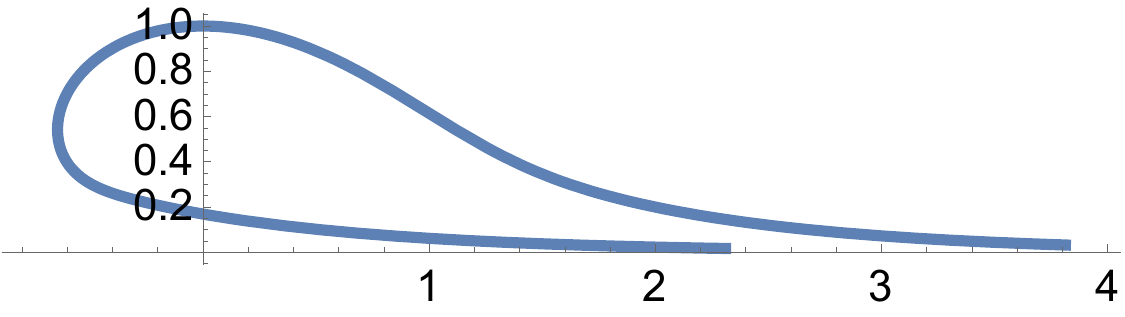}\quad \includegraphics[width=.3\textwidth]{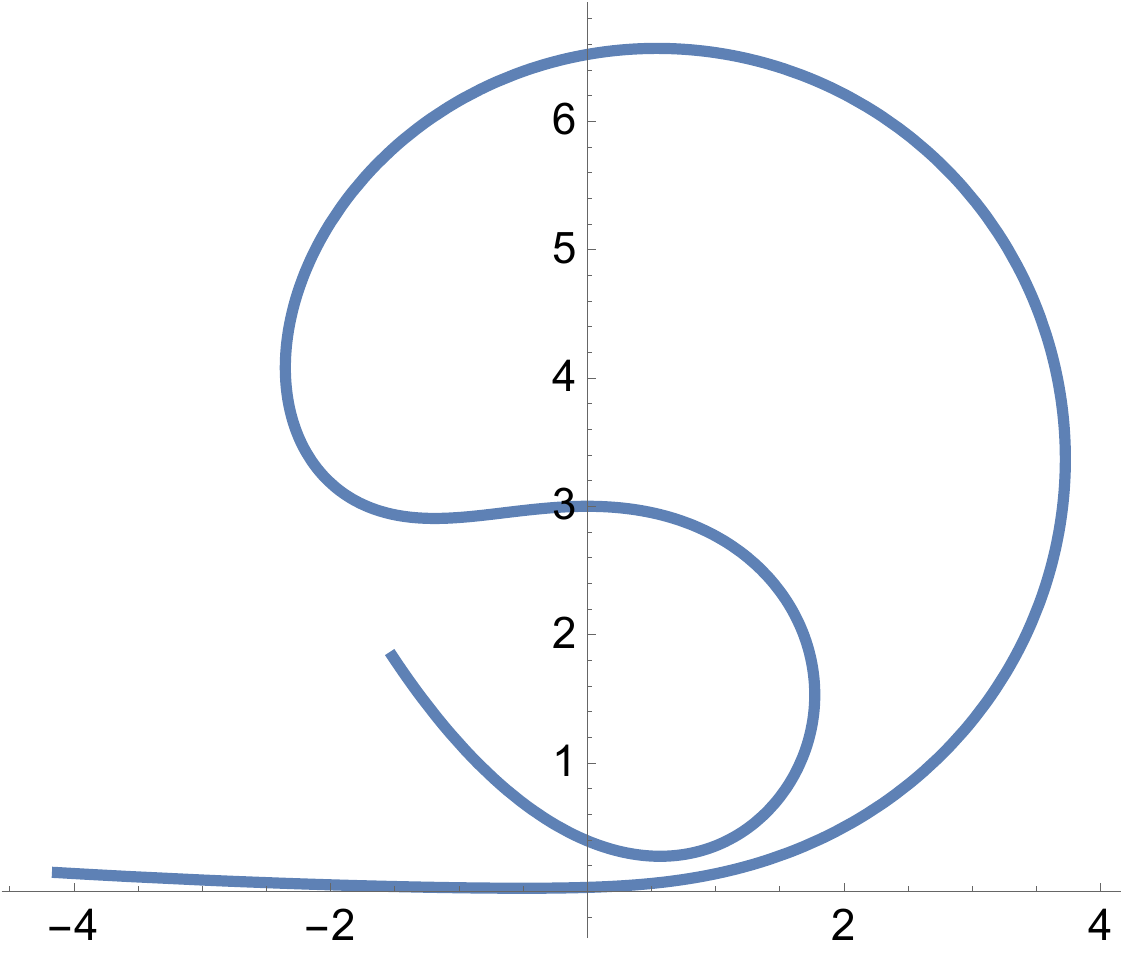}
\end{center}
\caption{Examples of solitons: hyperbolic (left), parabolic (middle) and rotational (right).}\label{fig1}
\end{figure}

Solitons from the Prop. \ref{prs} have appeared in the literature when studying the mean curvature flow of surfaces in hyperbolic spaces of arbitrary dimension. See \cite{brl,lrm}.

\section{Geodesic translating solutions}
\label{sec3}

When we consider each one of the three types of one-parameter subgroups of isometries $\mathcal{H}$, $\mathcal{P}$ and $\mathcal{R}$ defined in Sect. \ref{sec2} and we apply them to a given point of the hyperbolic plane, the orbit of this point is not a geodesic in general. In fact, we obtain, respectively, equidistant lines or geodesics,  horocycles and circles.
 
In this section we will employ a   suitable notion of a translations in $\h^2$ where points move along geodesics. Given a point $p \in \h^2$ and a tangent vector $V \in T_p \h^2$,  we define \textit{the geodesic translation} of  $p$ in the direction $v$ as a map $p \mapsto \eta(1;p,V)$,  where $\eta(t)=\eta(t;p,V)$ is the geodesic which at time $t = 0$ passes through $p$ with speed $V$. We point out that   geodesic translations are not isometries of the hyperbolic plane.

In order to compute   explicit parametrizations of these geodesics, we   relax the condition that $|V|=1$.  Depending on $V$, the parametrizations of the geodesics $\eta$ are the following. Suppose $p= (x,y) \in \h^2$ and $V=(v,w)$. Then we have  the following parametrizations:
\begin{enumerate}
\item Case $V$ is vertical. If $V=(0,w)$, then  $\eta(t) = (x, y e^{ w t})$. 
\item Case $V$ is horizontal. If $V=(v,0)$, then $
  \eta(t) = (x + y \tanh(vt), \frac{y}{\cosh(vt)})$. 
  \item General case of $V$. If $V=(v,w)$, then
 \begin{equation*}
 \begin{split}
 \eta(t) &= (a + R\tanh(t+t_0), \frac{R}{\cosh(t+t_0)}),\\
a& = x + \frac{w}{v}y,\,  R = \frac{y \sqrt{v^2+w^2}}{v},\, t_0 = \cosh^{-1} \left( \frac{R}{y}\right).
\end{split}
\end{equation*}
 \end{enumerate}
In what follows, we will compute all curves that evolve by geodesic translations in the hyperbolic plane. Without loss of generality, we can assume that $V=(0,1)$ (vertical direction), $V=(1,0)$ (horizontal direction), $V=(v,w)$, $v^2+w^2=1$ (arbitrary direction). 

 
\begin{proposition}
\label{prop:A} 
Let $\gamma(s)=(x(s),y(s))$  be a parametrized curve in $\h^2$ that under CSF evolves by geodesic translations along $V=(0,1)$.   Then  $\gamma$ is a horocycle, an equidistant line or a geodesic and the flow is given by $\gamma(t,s) = (as+c, (bs+d)e^t)$, $a,b,c,d\in\r$ and  $a^2+b^2\not=0$.  
\end{proposition}

\begin{proof} Since $\gamma$ evolves by geodesic translations along $V = (0,1),$ its flow is given by $$\gamma(t,s) = (x(s), y(s)e^t)$$. 
 By direct computation we get 
\begin{equation*}
\begin{split}
\partial_t \gamma(t,s) &= (0,  e^{t}y) \\
N(t,s) &= \frac{ye^t}{\sqrt{e^{2 t} y'^2+x'^2}}\left(- e^{  t}y',x' \right)\\
\kappa(t,s)&=\frac{e^{2 t} \left(x' \left(y y''+y'^2\right)-y x'' y'\right)+x'^3}{\left(e^{2 t} y'^2+x'^2\right)^{3/2}}
\end{split}
\end{equation*}
After some computations, the CSF equation \eqref{flow} is 
\begin{equation*}
 x'' y'-x' y''=0.
\end{equation*}
By \eqref{ke0}, this implies that the Euclidean curvature $\kappa_e$ is identically $0$, hence $\gamma$ is a straight-line. Therefore $\gamma$ is curve with constant curvature, that is,   a horocycle, an equidistant line or a geodesic. In any case, the parametrization of $\gamma$ is $\gamma(s)=s(a,b)+(c,d)$, where $a^2+b^2\not=0$. 
\end{proof}

  \begin{proposition} 
   Let $\gamma(s)=(x(s),y(s))$  be a parametrized curve in $\h^2$ that under CSF evolves by geodesic translations along $V=(1,0)$.   Then $\gamma$ is an equidistant line of curvature $\frac{1}{\sqrt{2}}$  and the flow is given by 
    \begin{equation}\label{t1}
  \gamma(t,s)=\left(s+(m-s)\tanh(t),\frac{m-s}{\cosh(t)}\right),
  \end{equation}
 where $m\in\r$ is an integration constant.  Furthermore, 
 \begin{enumerate}
\item the flow is defined for all $t \in \r$, where if $t \to \infty,$ then $\gamma_t$ approaches the geodesic $s\mapsto (m,s)$
and  if $t \to - \infty,$ then $\gamma_t$ approaches the ideal boundary $\h^2_\infty$;
\item for fixed $t$, the curve $\gamma_t$ is a equidistant line of curvature $\kappa=\tanh(t)$. 
\end{enumerate}
 
 \end{proposition}

 \begin{proof}
Since $\gamma$ evolves by geodesic translations along $V = (1,0),$ its flow is given by 
$$\gamma(t,s) = (x(s) + y(s) \tanh(t), \frac{y(s)}{\cosh(t)}).$$ 
Then $\gamma_t$ is the straight-line of equation $Y=m(X-x)$, where $m=1/\sinh(t)$. Thus its curvature is $1/\sqrt{1+m^2}=\tanh(t)$. By computing all the quantities in \eqref{flow}, we get 
\begin{equation*}
\begin{split}
\partial_t \gamma(t,s) &=\text{sech}(t) \left(y\ \text{sech}(t),-y \tanh (t)  \right) \\
N(t,s) &=\frac{y\text{sech}(t)}{\sqrt{2 \tanh (t) x' y'+x'^2+y'^2}}  \left(-\text{sech}(t)y' ,\tanh (t) y'+x'   \right)\\
\kappa(t,s)&=\frac{1}{\left(2 \tanh (t) x' y'+x'^2+y'^2\right)^{3/2}}\Big( -y \text{sech}^2(t) x'' y'+3 \tanh (t) x'^2 y'\\
&+x' \left(y \text{sech}^2(t) y''+\left(3-2 \text{sech}^2(t)\right) y'^2\right)+\tanh (t) y'^3+x'^3\Big)
\end{split}
\end{equation*}
Then \eqref{flow} becomes
$$e^{2 t} \left(x'+y'\right)^3+\left(x'-y'\right)^2 \left(x'+y'\right)+2 y \left(x' y''-x'' y'\right)=0.$$
Looking this equation on $e^{2t}$, we deduce  
\begin{equation*}
\begin{split}
0&= x'+y'\\
0&= \left(x'-y'\right)^2 \left(x'+y'\right)+2 y \left(x' y''-x'' y'\right).
\end{split}
\end{equation*}
From the first equation, we have  $x(s)+y(s) = m,$ for some constant  $m$. Now, the second equation implies $x' y''-x'' y'=0$, that is, $\gamma$ is a straight-line, which we can parametrize by $y(s)=-s+m$.  So, it follows that $\gamma$ parametrizes an equidistant line of curvature $\frac{1}{\sqrt{2}}$ and the flow is \eqref{t1}. 
The properties   (1) and (2) are immediate.

 \end{proof}
 
\begin{proposition}
Let $\gamma(s) = (x(s), y(s))$ be a parametrized curve in $\h^2$ that under CSF evolves by geodesic translations along $V=(v,w)$,  $v,w\not=0$.   Then $\gamma$ is an equidistant line of curvature $\kappa=\frac{v}{\sqrt{2(1-w)}}$ and the flow is given by 
\begin{equation}\label{t2}
\gamma(t,s)= \left(\frac{m \left(e^{2 t}-1\right) v+2 s}{-e^{2 t} (w-1)+w+1},\frac{2 e^t \left(m+\frac{s (w-1)}{v}\right)}{-e^{2 t} (w-1)+w+1}\right),
\end{equation}
 where $m\in\r$ is an integration constant.  Furthermore, 
 \begin{enumerate}
\item the flow is defined for all $t \in\r$, where if $t \to \infty,$ then $\gamma_t$ approaches the geodesic $s\mapsto (m,s)$
and  if $t \to - \infty,$ then $\gamma_t$ approaches the ideal boundary $\h^2_\infty$;
\item for fixed $t$, the solution $\gamma_t$ parametrizes an equidistant line of curvature $\kappa=\frac{v}{\sqrt{v^2+e^{2t}(w-1)^2}}$. 
\end{enumerate}\end{proposition}

\begin{proof}
 Geodesic translations in the direction $(v,w)$ of $\gamma$ have   the form $$\gamma(t,s) = (a + R   \tanh(t-t_0), \frac{R}{\cosh(t - t_0)}),$$ where 
$$a  = x(s) + \frac{w}{v}y(s),\quad  R= y(s) \cosh(t_0), \quad  t_0 = \tanh^{-1} (w).$$
Then 
$$\gamma(t,s)=\left(\frac{\left(e^{2 t}-1\right) v y(s)}{-e^{2 t} w+e^{2 t}+w+1}+x(s),\frac{2 e^t y(s)}{-e^{2 t} w+e^{2 t}+w+1}\right).$$
By computing all the quantities in \eqref{flow}, we get an expression of the type
$$A_0(s)+A_1(s)e^{2t}+A_2(s)e^{4t}=0,$$
where
\begin{equation*}
\begin{split}
A_2(s)&=\left((w-1) x'-v y'\right)^3,\\
A_1(s)&=2 \Big(x' \left(\left(-3 v^2 w+v^2+2 w-2\right) y'^2+(w-1) y y''\right)+v \left(v^2-2\right) y'^3\\
&+v \left(3 w^2-2 w-1\right) x'^2 y'-(w-1) y x'' y'-(w-1)^2 (w+1) x'^3\Big),\\
A_0(s)&=-v^3 y'^3+x' \left(v^2 (3 w+1) y'^2-2 (w+1) y y''\right)\\
&+v \left(-3 w^2-2 w+1\right) x'^2 y'+2 (w+1) y x'' y'+(w-1) (w+1)^2 x'^3.
\end{split}
\end{equation*}
Therefore, all coefficients $A_n$ must vanish identically. From $A_2=0$, we get 
\begin{equation}\label{yy}
y'(s)=\frac{w-1}{v}x'(s).
\end{equation}
If $w = 1,$ then $y$ is a constant function and hence the initial curve $\gamma$ parametrizes a horocycle, but this is impossible since horocycle evolve by geodesic translations along $(0,1).$ Therefore, $w \neq 1.$

By substituting the expression of $y$ given in \eqref{yy}  in   $A_1$ and $A_0$, then $A_1$ and $A_0$ vanish trivially.  Definitively, by solving \eqref{yy}, the curve $\gamma$ is a straight-line. By putting $x(s)=s$, then  
$$y(s)=\frac{w-1}{v}s+m,\quad m\in\r.$$
Since the slope of this Euclidean straight-line is $\frac{w-1}{v}$, then $\gamma$  is an equidistant line of curvature $\frac{v}{\sqrt{2(1-w)}}$. By substituting $\gamma$ in the expression of $\gamma(t,s)$, we obtain \eqref{t2}.  Notice that  the curve $\gamma_t$ parametrizes a Euclidean straight-line of slope $\frac{e^t (w-1)}{v}$.

\end{proof}
\section{The angle as parameter in the flow}
\label{sec4}
 
Besides solitons and translating solutions   studied in the previous sections, another interesting class of solutions are convex ancient solutions. Those are the solutions of the CSF whose time interval is of the form $J = ( -\infty, T_0 ).$ In \cite{dhs}, the authors have been able to classify compact convex ancient solutions in the Euclidean plane by introducing \textit{the pressure function} $p = \kappa^2$ and then proving the following result:

\begin{theorem}[Theorem $1.1$ in \cite{dhs}]
Let $p(t, \theta) = \kappa^2(t, \theta)$ be the pressure function corresponding to a compact convex ancient solution of the curve shortening flow. Then $p$ is of the form $$p(t, \theta) = a(t) + b(\theta),$$ for some functions $a,b$.
\end{theorem}  
It is of our interest to examine to what extent is this theorem true in the hyperbolic plane. In order to do so, first we need to see what parameter can fill the role of the turning angle $\theta$ in the Euclidean plane and secondly, we will find all convex solutions whose pressure function can be written by separating the variables.

When considering convex curves in the Euclidean plane, one can use the turning angle $\theta$ for the parametrization which it is highly advantageous because of the geometric importance of such parameter. Because $\theta$ was considered as a parameter in \cite{dhs} (see also \cite{blt}), and since the hyperolic metric in the upper-half plane model   is conformal to the Euclidean one in $\r^2_{+}$, we will compute the evolution equations for the Euclidean angle $\theta$ in $\h^2.$ However, because the turning angle is closely related only to the Euclidean plane, we do not expect the situation to be the same in $\h^2,$ as was indicated already by Grayson in Sect. 2 of \cite{gr}. With this purpose, Grayson   introduced a parameter $\varphi$ which is defined purely analytically and which can be used instead of the turning angle for curves on arbitrary surfaces. Grayson defines globally an angle $\varphi=\varphi(t,s)$. First, the function $\varphi=\varphi(t,u)$ is defined on the initial curve $\gamma(s)=\gamma(0,s)$ by 
 $$\varphi(0,u)=\int_{0}^{s(u)}\kappa\,ds.$$
 Along $\gamma(s)$, this function $\varphi(0,s)$ coincides with the angle of $\gamma'(s)$ with a fixed direction of $\h^2$. Next, a correction factor $\rho$ is defined by the equation $\rho(\partial\varphi/\partial s)=\kappa$. Finally, the angle function $\varphi$ is defined by  the equation 
$$\rho\frac{\partial\varphi}{\partial t}=\frac{\partial \kappa}{\partial s}.$$
The next step consists of modifying the initial flow to preserve tangent directions. To this end, the parameter $t$ is replaced by $\tau$, where
$$\frac{\partial}{\partial\tau}=\frac{\partial}{\partial t}-\frac{1}{\kappa}\frac{\partial\kappa}{\partial s} \frac{\partial}{\partial s}.$$
  This allows us to obtain the next identity for the evolution of curvature when $\varphi$ is fixed: 
  \begin{equation}\label{eq3}
\kappa_\tau= \kappa^2 \kappa_{\varphi \varphi} + \kappa^3 - \kappa,
\end{equation}
 where we write the partial derivatives as subindices.  See \cite[Lem. 2.7]{gr}. We take into account that $\tau$ and $\varphi$ are independent variables because  it was proved in \cite[Lem.2.1]{gr} that 
$\partial\varphi/\partial \tau=0$. When using $\varphi$ as a parameter, the evolution equation \ref{eq3} is analogous to the one in the Euclidean plane (see \cite{gh}).

Define the following functions: 
$$p = \kappa^2,\quad  q = p_{\varphi}.$$

\begin{lemma} 
 The evolution equations for $p$ and $q$ are 
\begin{equation}\label{evo}
\left\{
\begin{split}
p_\tau &= p p_{\varphi \varphi} - \frac{1}{2}p_{\varphi}^2 + 2p^2 - 2p \\
 q_\tau&= p(q_{\varphi \varphi} + 4q) - 2 q
\end{split}\right.
\end{equation}
\end{lemma}

\begin{proof}
We have $p_\varphi=2\kappa\kappa_\varphi$, so $\kappa_\varphi=p_\varphi/(2\kappa)$. This implies
$$p_{\varphi\varphi}=\frac{p_\varphi^2}{2p}+2\kappa\kappa_{\varphi\varphi},$$
hence
$$\kappa_{\varphi\varphi}=\frac{1}{2\kappa}\left(p_{\varphi\varphi}-\frac{p_\varphi^2}{2p}\right).$$
Then using \eqref{eq3}, we have
$$p_\tau=2\kappa\kappa_\tau=2\kappa\left(\frac{p}{2\kappa}(p_{\varphi\varphi}-\frac{p_\varphi^2}{2p})+\kappa^3-\kappa\right)=p p_{\varphi \varphi} - \frac{1}{2}p_{\varphi}^2 + 2p^2 - 2p .$$
For the computation of $q_\tau$, we use that $\tau$ and $\varphi$ are independent variables.   This implies 
\begin{equation*}
\begin{split}
q_\tau&=(p_\varphi)_\tau=(p_\tau)_\varphi\\
&=\frac{\partial}{\partial\varphi}\left(pp_{\varphi\varphi}-\frac12p_\varphi^2+2p^2-2p\right)\\
&=pp_{\varphi\varphi\varphi}+4pp_\varphi-2p_\varphi=p(q_{\varphi\varphi}+4q)-2q.
\end{split}
\end{equation*}

\end{proof}

From now on, we replace the notation $\tau$ by $t$, such as Grayson suggests in \cite{gr} (see also \cite{gh}). However, we keep the variable $\varphi$ because we will study later the angle $\theta$. Because the curvature determines the curve, it is sufficient to solve the evolution equation for $p$. In order to do so, we consider the solutions that satisfy $q_t = 0$. Then $p_{\varphi t} = 0$ which implies that the function $p=p(t,\varphi)$ is of the form $  a(t) + b(\varphi)$ for some functions $a=a(t)$ and $b=b(\varphi)$. 

\begin{theorem}\label{t42}
Let $\{\gamma_t\colon t \in J\}$ be a convex solution to the CSF in $\h^2$ such that the corresponding pressure function $p$ can be written as 
$$p(t, \varphi) = a(t) + b(\varphi),$$
for some functions $a,b$. Then the solution   is a  family of shrinking circles, horocycles, equidistant lines or solitons.
\end{theorem}

\begin{proof}
Equations \eqref{evo} become
\begin{equation}\label{evo2}
\left\{
\begin{split}
a' &= (a+b)b'' - \frac{1}{2}b'^2 + 2(a+b)^2 - 2(a+b) \\
0&= (a+b)(b'''+4b')-2b',
\end{split}\right.
\end{equation}
where $'$ denotes, in each case, the corresponding derivative with respect to the variable $t$ or $\varphi$. 
 If the function $a(t)=a$ is constant, then at any time $t$, the curves $\gamma(t, \cdot)$ and $\gamma(0, \cdot)$ have the same curvature function, namely $\kappa(t, \cdot) = \sqrt{a + b(\cdot)} = \kappa(0, \cdot)$. 
Because two curves have the same curvature, there exists an isometry of the ambient space $\Phi(t)$ such that $\gamma(t, \cdot) = \Phi(t)\gamma(0, \cdot)$. Since $t$ is an arbitrary time, the solution evolves by isometries,  that is, $\gamma$ is a soliton.  
 
Now assume that $a$ is not a constant function. We distinguish two cases.
\begin{enumerate}
\item Assume   $b(\varphi) $ is a constant function. Without loss of generality, we can assume that this constant is $0$. Then $p(t, \varphi) = a(t)$.  The second equation in \eqref{evo2} is trivial and the first one   reduces to  
$$a' = 2a^2 - 2a.$$ 
We now solve this equation. This equation can be written as
$$\frac{da}{a(a-1)}=2\, dt,$$
obtaining
$$a(t)=\frac{1}{1+C e^{2t}},\quad C\in\r.$$
Since the function $a$ is positive, we have three types of equations depending on the sign of the integration constant $C$:
\begin{equation}\label{a3}
a(t) = \left\{\begin{array}{lll} 
 &\frac{1}{1-Ce^{2t}},  & t<-\frac12\log(C), C>0\\ 
  & 1 &t\in\r\\
  &\frac{1}{1+Ce^{2t}},&t\in\r,  C>0.
  \end{array}\right.
\end{equation}
\begin{enumerate}
\item Case $a(t) =  \frac{1}{1-Ce^{2t}}$, where $t\in (-\infty,-\frac12\log(C))$ and $C>0$. Let $c=\frac12\log(C)$. Then $C=e^{2c}$. Define
$$r(t)=\cosh^{-1}(e^{-(t+c)}).$$
 Then we have 
$$\kappa(t, \varphi) = \sqrt{p(t, \varphi)} = \sqrt{a(t)} = \frac{1}{\tanh(r(t))}.$$ 
This implies that the  corresponding solution is a  family of shrinking circles.  
 \item Case  $a(t)=1$, $t\in\r$. Then $\kappa=1$. Thus $\gamma(t,\cdot)$ is a horocycle for all $t$. 

\item Case $a(t) =  \frac{1}{1+Ce^{2t}}$, where $t\in \r$ and $C>0$. Again, if $c=\frac12\log(C)$, let   
$$r(t)=\sinh^{-1}(e^{t+c}).$$
 Then
$$\kappa(t, \varphi) = \sqrt{p(t, \varphi)} = \sqrt{a(t)} = \frac{1}{\cosh(r(t))}.$$ 
In consequence, the solution is a family of equidistant lines. Notice that $k<1$.
 \end{enumerate}

\item Assume $b$ is not a constant function.   The second equation of \eqref{evo2} is  
\begin{equation}\label{e3}
(a+b)(b''' + 4b') - 2b' = 0.
\end{equation}
If $b'''+4b'=0$ identically, then \eqref{e3} implies $b'=0$. Thus $b$ is a constant function, which it is a contradiction. Thus  $b'''+4b'\not=0$.  From \eqref{e3}, we have 
 $$a(t) = \frac{2b'(\varphi)}{b''' (\varphi)+ 4b'(\varphi)} - b(\varphi).$$ 
The left-hand side of the equation above depends only on $t$ while the right-hand side depends only on $\varphi$. In particular, the function $a$ is constant. This contradiction proves that this case is not possible.  
\end{enumerate}
\end{proof}

\begin{corollary}
Let $\{\gamma_t\colon t \in J\}$ be a   closed  convex solution to the CSF in $\h^2$ such that the corresponding pressure function $p$ can be written as 
$$p(t, \varphi) = a(t) + b(\varphi),$$
for some functions $a,b$. Then the solution   is a  family of shrinking circles.
\end{corollary}

\begin{proof} According to the Thm. \ref{t42}, the cases of horocycles and equidistant lines cannot occur, as well as of solitons because there are no compact solitons. 
\end{proof}

We finish this section with obtaining the evolution equation of the  Euclidean angle $\theta$ when $\h^2$ is viewed in the upper half-plane model. If $\gamma=\gamma(t,s)$, let   $\gamma(t,s)=(x(t,s),y(t,s))$. Define  $\theta=\theta(t,s)$ be the Euclidean angle given by  
\begin{equation}\label{da}
T(t,s)=y(t,s)(\cos\theta(t,s),\sin\theta(t,s)).
\end{equation}

\begin{lemma} Let $\gamma(t,s)$ be  a       solution to the CSF \eqref{flow}. If $\theta$ is the Euclidean angle defined in \eqref{da}, then 
\begin{equation}\label{ea}
\begin{split}
\frac{\partial \theta}{\partial t}& = \kappa_s + \kappa \sin \theta,\\
\frac{\partial \theta}{\partial s} &= \frac{\kappa-\cos\theta}{y}.
\end{split}
\end{equation}
\end{lemma}

\begin{proof} 
The Levi-Civitta connection of $\h^2$ in the upper half-plane model is
$$\nabla_{\partial_x}\partial_x = \frac{1}{y}\partial_y,\quad
\nabla_{\partial_x}\partial_y = -\frac{1}{y}\partial_x,\quad
\nabla_{\partial_y}\partial_x = -\frac{1}{y}\partial_x,\quad
\nabla_{\partial_y}\partial_y =- \frac{1}{y}\partial_y.$$
We have
\begin{equation}\label{T1}
\frac{\partial T}{\partial t}=\frac{\partial}{\partial
t}\frac{\partial\gamma}{\partial s}=
\frac{\partial }{\partial s}(\kappa
N)+\kappa^2T=\frac{\partial\kappa}{\partial s}N.
\end{equation}
On the other hand,
\begin{equation*}
\begin{split}
\frac{\partial T}{\partial
t}&=\nabla_{\partial_t}(y\cos\theta\partial_x+y\sin\theta \partial_y)\\
&=(y_t\cos\theta-y\theta_t\sin\theta)\partial_x+(y_t\sin\theta+y\theta_t\cos\theta)\partial_y\\
&+y\cos\theta\nabla_{\partial_t}\partial_x+y\sin\theta\nabla_{\partial_t}\partial_y.
\end{split}
\end{equation*}
We know $\partial_t=\gamma'(t)=x_t\partial_x+y_t\partial_y$. Thus
$$\nabla_{\partial_t}\partial_x=-\frac{y_t}{y}\partial_x+\frac{x_t}{y}\partial_y.$$
$$\nabla_{\partial_t}\partial_y=-\frac{x_t}{y}\partial_x-\frac{y_t}{y}\partial_y.$$
Therefore,
\begin{equation}\label{T2}
\frac{\partial T}{\partial
t}=(y\theta_t+x_t)\left(-\sin\theta\partial_x+\cos\theta\partial_y\right)=
\frac{y\theta_t+x_t}{y}N=(\theta_t+\frac{x_t}{y})N.
\end{equation}
On the other hand,  from \eqref{flow} we have
$$\frac{\partial\gamma}{\partial t}=(x_t,y_t)=\kappa N=\kappa
y(-\sin\theta,\cos\theta).$$
This gives 
\begin{equation}\label{xy} x_t=-\kappa y\sin\theta,\quad y_t=\kappa y\cos\theta.
\end{equation}
 Using the value of $x_t$ and equating \eqref{T1} and \eqref{T2}, we obtain the first equation of \eqref{ea}. 

Finally, we know that $\theta_s=\kappa_e$ is the Euclidean curvature of $\gamma$.  Then the relation \eqref{kv}   becomes now $$\kappa=y\kappa_e+\cos\theta=y\frac{\partial\theta}{\partial s}+\cos\theta,$$
hence we get the second equation of \eqref{ea}.
\end{proof}

\begin{theorem} \label{t46}
Let $\gamma(t,s)$ be  a       solution to the CSF \eqref{flow}. If $\theta$ is the Euclidean angle defined in \eqref{da}, then 
\begin{equation}\label{ss}
(y-1)\left(k_{ss}+k_s \sin\theta\right)=0. 
\end{equation}
In particular, if $\gamma$ is not the horocycle $y=1$, we have 
$$k_{ss}+k_s \sin\theta=0. $$
\end{theorem}
\begin{proof}
We use the formula for the Lie bracket  $[\partial_t, \partial_s] = \kappa^2 \partial_s$: see \cite{gr}. Thus $\theta_{ts}=\theta_{st}-\kappa^2\theta_s$. Using this in combination with   \eqref{ea} and \eqref{xy}, we have
\begin{equation*}
\begin{split}
\kappa_{ss}&=\frac{\partial}{\partial s}\frac{\partial\theta}{\partial t}-\kappa_s\sin\theta-\cos\theta\kappa\theta_s\\
&=\frac{\partial}{\partial t}\frac{\partial\theta}{\partial s}-\kappa^2\theta_s-\kappa_s\sin\theta-\cos\theta\kappa\theta_s\\
&=\frac{\partial}{\partial t} \frac{ \kappa-\cos\theta}{y}-\kappa^2\theta_s-\kappa_s\sin\theta-\cos\theta\kappa\theta_s\\
&=-\frac{y_t}{y^2}(\kappa-\cos\theta)-\frac{\kappa_t+\sin\theta\kappa_t}{y}-\kappa(\kappa+\cos\theta)\theta_s-\sin\theta\kappa_s\\
&=\frac{1}{y}\left(-\kappa^3+\kappa+\kappa_t+(1-y)\sin\theta\kappa_s\right).
\end{split}
\end{equation*}
This implies 
$$y\kappa_{ss}+\kappa^3-\kappa=\kappa_t+(1-y)\sinh\theta\kappa_s.$$
To finish the proof, we use in this identity the expression of $\kappa_{t}$ given by 
\begin{equation}\label{dgr}
\kappa_t=\kappa_{ss}+\kappa^3-\kappa.
\end{equation}
See \cite[Lem. 1.3]{gr}. This yields \eqref{ss}.
\end{proof}

Thanks to this theorem, we obtain a similar result that Thm. \ref{t42} for the Euclidean angle   $\theta$.

\begin{corollary}\label{c47} 
Let $\{\gamma_t\colon t \in J\}$ be a convex solution to the CSF in $\h^2$ such that $\kappa(t,s)=a(t)+b(s)$ for some functions $a$ and $b$.  Then the solution   is a  family of shrinking circles, horocycles, equidistant lines or solitons.
 \end{corollary}

\begin{proof}
Suppose that $\gamma$ is not a horocycle. Then \eqref{ss} implies 
$ b''+b'\sin\theta=0$.  We separate two cases.
\begin{enumerate}
\item Case $b'=0$. Let $b(s)=b_0$. This implies that $\kappa$ depends only on $t$. Then \eqref{dgr} is 
$a'=(a+b_0)^3-(a+b_0)$. Renaming the function $a(t)$ as $a(t)+b_0$, the solution of this equation is 
$$a(t)=\frac{1}{\sqrt{1+Ce^{2 t }}},\quad C\in\r.$$
Then $\kappa(t)$ coincides as \eqref{a3}, obtaining  that  the solution   is a  family of shrinking circles, horocycles or  equidistant lines. 
\item Case $b'\not=0$. From $ b''+b'\sin\theta=0$ we deduce that $\theta$ depends only on $s$. Then the first equation of \eqref{ea} implies $b'+(a+b)\sin\theta=0$. Since $b'\not=0$, then $\sin\theta\not=0$, hence $a$ is a constant function. Thus $\gamma$ is a soliton of the CSF.
\end{enumerate}
\end{proof}

\section{An estimate of the area for ancient solutions of CSF}\label{sec5}

Let $\{ \gamma_t \colon t \in \langle -\infty, T_0 \rangle \}$ be an ancient solution to the CSF \eqref{flow}. If $T_0 < \infty$,  then we can rescale the flow so that $T_0 = 0$. Therefore, without the loss of generality, we can assume that either $T_0 = 0$ or $T_0 = \infty$. Ancient convex solutions to the CSF have been classified both in the Euclidean plane and on the sphere  \cite{blt,bl}.So far, the only explicit ancient convex solutions in $\h^2$ are contracting circles, geodesics, equidistant lines and horocycles.

In the   case of simply closed solutions, we obtain the following computation of the area $A(t)$ along the evolution, which are motived by the results in \cite{bl}.  
\begin{theorem}
\label{t51}
Let $\{\gamma_t\colon t \in J\}$ be an   ancient simply closed solution to the CSF in $\h^2$. Then $\displaystyle \lim_{t \to - \infty}A(t) = + \infty$,  where $A(t)$ denotes the area enclosed by $\gamma(t, \cdot)$. 
\end{theorem}
\begin{proof}
 The Gauss-Bonnet theorem in $\h^2$ says  $\int_{\gamma_t} \kappa(t,\cdot) \, ds = 2\pi + A(t)$. We differentiate with respect to $t$. Here we use 
$$ \frac{d}{dt} ds = - \kappa^2 ds.$$
See \cite[p. 75]{gr}. By using this identity together \eqref{dgr},  we have 
\begin{equation*}
\begin{split}
\frac{dA}{dt} &= \frac{d}{dt} \int_{\gamma_t}\kappa(t,\cdot) \, ds =\int_{\gamma_t}(\kappa_{ss}+\kappa^3-\kappa)(t,\cdot)\, ds-\int_{\gamma_t}\kappa(\kappa^2)(t,\cdot)\, ds\\
&= - \int_{\gamma_t} \kappa(t,\cdot) \, ds = - 2\pi - A(t).
\end{split}
\end{equation*}
The solution of this equation is $A(t)=c e^{-t}-2\pi$, $c\in\r$. At $t=0$, we have $A_0:=A(0)=c-2\pi$. Thus $c=A_0+2\pi$, obtaining
 $$A(t) = (A_0 + 2\pi)e^{-t} - 2\pi.$$ If we let $t \to - \infty$,  we   get the claim.
\end{proof}

\begin{corollary}
Let $\{\gamma_t\colon t \in ( - \infty, T_0 ) \}$ be an   ancient simply closed solution  to the CSF in $\h^2$. Then $T_0 = 0$.
\end{corollary}
\begin{proof}
From the proof of the Thm. \ref{t51} we know that the area enclosed at the time $t$ is given by   $A(t) = (A_0 + 2\pi)e^{-t} - 2\pi$. If $T_0 = \infty$,  letting $t \to \infty$ we would get $A(t) \to -2\pi$ but that is impossible since the area must be non-negative. Therefore, it must be that $T_0 = 0$.
\end{proof}

\begin{remark} 
In \cite{gr}, it was proved    that if the maximal time $T_0$ is finite, then the curve collapses to a point. In consequence  $A_0 = 0$ and therefore the area enclosed at time $t$ is equal to  
\begin{equation}\label{aa}
A(t) =2\pi e^{-t} - 2\pi = 2\pi (e^{-t} - 1).
\end{equation}
\end{remark}

 \section*{Acknowledgements}
Rafael L\'opez has been partially supported by MINECO/MICINN/FEDER grant no. PID2023-150727NB-I00,  and by the ``Mar\'{\i}a de Maeztu'' Excellence Unit IMAG, reference CEX2020-001105- M, funded by MCINN/AEI/10.13039/ 501100011033/ CEX2020-001105-M.

\end{document}